\documentclass[preprint,12pt]{elsarticle}
\journal{Theoretical Computer Science}
\pdfoutput=1
\usepackage[english]{babel}
\usepackage{amsmath,amssymb,amsxtra,amsthm}
\usepackage{microtype}

\usepackage[charter]{mathdesign}
\usepackage{natbib}

\newtheorem{theorem}{Theorem}
\newtheorem{problem}{Problem}

\newtheorem{remark}[theorem]{Remark}

\theoremstyle{definition}
\newtheorem{example}{Example}

\newcommand{\uno}{\varphi}
\newcommand{\Uuno}{\Phi}
\newcommand{\Thue}{\mathbf{t}}

\begin{document}

\begin{frontmatter}

\title{Disposability in Square-Free Words}

\author{Tero Harju}
\address{Department of Mathematics and Statistics,
         University of~Turku, Finland}
\ead{harju@utu.fi}


\begin{abstract}
We consider words $w$ over the alphabet $\Sigma=\{0,1,2\}$. It is shown that there are
irreducibly square-free words of all lengths $n$ except 4,5,7 and 12. Such a word is
square-free (i.e., it has no repetitions $uu$ as factors), but by removing any one internal
letter creates a square in the word.
\end{abstract}

\end{frontmatter}

\section{Introduction}
Grytczuk et al.~\cite{Grytczuk} showed that there are infinitely many `extremal' square-free
ternary words where one cannot augment a single new letter anywhere without creating a square;
see also Mol and Rampersad~\cite{MolRampersad2020} for further results.
In this article we consider the dual problem of this and show that there are square-free ternary words
of all lengths, except $4,5,7$ and $12$, where removing any single interior letter creates
a square. Although the problems resemble each other, the results and the proof techniques are
quite different.

Let $\Sigma=\{0,1,2\}$ be a fixed ternary alphabet and denote by $\Sigma^*$ and
$\Sigma^\omega$ the sets of all finite and infinite length words over $\Sigma$,
respectively. A finite word $u$ is called a \emph{factor} of a word $w\in \Sigma^* \cup \Sigma^\omega$
if $w=w_1uw_2$ for some, possibly empty, words $w_1$ and $w_2$.
Moreover, $w$ is \emph{square-free} if it does not have a
nonempty factor of the form $uu$.

Let $w \in \Sigma^*$ be a square-free word with a factorization $w=w_1aw_2$ where $a \in \Sigma$.
We say that the occurrence of the letter $a$ is \emph{disposable} if $w_1w_2$ is square-free.
The definition extends naturally to infinite words. An  occurrence of a letter $a$ is \emph{interior}, if
$w_1$ and $w_2$ are both nonempty.

If a square-free word $w\in \Sigma^* \cup \Sigma^\omega$ does not have disposable occurrences of
interior letters
then $w$ is said to be \emph{irreducibly square-free}, i.e., by deleting any
interior occurrence of a letter results in a square in the remaining word.

The nonemptiness condition on
the prefixes and suffixes is required since all prefixes and suffixes of square-free words
are disposable.

\begin{remark}
The words of length at most two have no internal letters, and
therefore we consider the property of being
irreducibly square-free only
for words of length at least three.
\end{remark}

\begin{example}
Let $\tau\colon \Sigma^* \to \Sigma^*$ be the morphism determined by
\[
\tau(0)=012, \quad \tau(1)=02, \quad \tau(2)=1\,.
\]
The \emph{Thue word} $\Thue$ is the fixed point $\Thue=\tau^\omega(0)$ of $\tau$ obtained
by iterating $\tau$ on the start word $0$. Then $\Thue$ is an infinite square-free word; see, e.g.,
Lothaire~\cite{Lothaire}:
\[
\Thue=012 02 1 012 1 02 012 02 1 02 012 1 012 02 1 012 \cdots\,
\]
We show that the Thue word is \emph{not} irreducibly
square-free. For this, we first notice that $\Thue$ avoids $010$ and $212$ as factors. Also, it avoids
$1021$, since this word would have to be a factor of $\tau(212)$. Deleting the letter
$2$ at the third position results in a square-free word $01021012102012\cdots$.
Indeed,  a potential square would
have to start either from the beginning, but the prefix $010$ does not occur in $\Thue$, or from
the second position, but $1021$ does not occur in $\Thue$. 
\end{example}

\medskip

Later checking of irreducibility of (infinite) words is based on the following procedure
that depends on a morphism $\alpha\colon \Sigma^* \to \Sigma^*$
for which $|\alpha(a)| > 1$ for all letters $a$.

\medskip

\noindent \textbf{Procedure~I.}
\begin{itemize}
\item[1]
Check that the morphism $\alpha$ generates an infinite
square-free word; say, $\alpha^\omega(0)$ or $\alpha(w)$, where $w$ is a given
infinite square-free word.

\item[2]
For any pair $(a,b)$ of different letters, check that $\alpha(ab)$ is irreducibly square-free.
This takes care that the last letter of $\alpha(a)$ and the first letter of $\alpha(b)$ are
not disposable in $\alpha(ab)$. This guarantees that these occurrences are not disposable
in any $\alpha(w)$ where $w=w_1abw_2$ is square-free.
\end{itemize}

The first item of Procedure~I is often taken care of by Crochemore's criterion ~\cite{Crochemore}:

\begin{theorem}\label{Crochemore}
A morphism $\alpha\colon \Sigma^* \to \Sigma^*$ preserves  square-free words if and
only if it preserves square-freeness of words of length five.
\end{theorem}

\section{Irreducibly square-free words of almost all lengths}

By a systematic search we find that there  are no irreducibly square-free words of
lengths 4,5,7 and 12. In the following table we have counted the irreducibly
square-free words of lengths $3, \ldots, 30$ up to isomorphism (produced by permutations of
the letters) and reversal (mirror image) of the words. For instance, $010212010$ is
the only irreducibly square-free word of length nine up to isomorphism and reversal.
It is a palindrome.
The table suggests that the irreducibly square-free words are quite rare among the
square-free words, e.g., there are (up to isomorphism and reversal) 202 square-free words
of length 20, but only 12 of those are irreducibly square-free.
Counting the numbers of (irreducibly) square-free words must take into consideration
those words that are palindromes or isomorphic to their reversals.

 \begin{table}[htb]
 \begin{center}
\begin{tabular}{|rr|rr|rr|rr|rr|rr|rr|rr|}\hline
length & card &&&&&&&&&&&&\\ \hline
3& 1 &
4& 0 &
5& 0 &
6& 1 &
7& 0       &
8& 1       &
9& 1      \\
10& 1     &
11& 3     &
12& 0     &
13& 3     &
14& 4     &
15& 4     &
16& 7     \\
17& 9     &
18& 7     &
19& 12   &
20& 12   &
21& 16   &
22& 18   &
23& 23   \\
24& 24  &
25& 34  &
26& 36  &
27& 48  &
28& 55  &
29& 69  &
30& 78  \\
\hline
\end{tabular}
\end{center}
           \caption{The number of irreducibly square-free words of lengths from 3 to 30
           up to isomorphism and reversal.}
\end{table}

\begin{theorem}\label{thm:2}
There exists an infinite irreducibly square-free word.
\end{theorem}

\begin{proof}
Let $\uno$ be the following uniform palindromic morphism of length 17,  i.e., $\uno(1) = \pi(\uno(0))$ and $\uno(2)=\pi^2(\uno(0))$
for the permutation $\pi=(0\ 1\ 2)$ of the letters:
\begin{align*}
\uno(0) &=01202120102120210\\
\uno(1) &=12010201210201021\\
\uno(2) &=20121012021012102
\end{align*}
By Theorem~\ref{Crochemore}, $\uno$ preserves square-freeness. It is easy to
check that $\uno(0)$, and so also the isomorphic copies $\uno(1)$ and $\uno(2)$, are
irreducibly square-free. Finally, Procedure~I entails that deleting the `middle' 17th
letters $0$ of $\uno(01)$ and of $\uno(02)$ gives squares: $11$ and $02120212$,
respectively. Similarly, deleting the 18th letter of $\uno(01)$ and of $\uno(02)$ gives
squares: $10201020$ and $00$, respectively. These observations suffice for the proof of
the theorem, since now $\uno(w)$ is irreducibly square-free for \emph{all} square-free,
finite or infinite, words $w$.
\end{proof}

\begin{remark}
The morphism $\uno$ has an \emph{alignment property}, i.e., 
for all letters $a,b,c$ if $\uno(bc)=u\uno(a)v$ then
$u$ or $v$ is empty, and $a=b$ or $a=c$, respectively.
\end{remark}

The morphism $\uno$ has an infinite fixed point
\[
\Uuno=\uno^\omega(0)
\]
that is the limit of the sequence $\uno(0), \uno^2(0), \ldots$.

Note that the finite prefixes of $\Uuno$ are not always
irreducibly square-free. For instance, none of the prefixes of $\Uuno$ of length $n$ with $19\le n \le 29$
are irreducibly square-free.
However, we do have the following result with the help of $\uno$.

\begin{theorem}\label{thm:3}
There are irreducibly square-free words of all lengths $n$ except 4,5,7 and 12.
\end{theorem}

\begin{proof}
Table~\ref{tab:small} gives an example for the cases $n \le 17$.

\begin{table}[hbt]
\begin{center}
\begin{tabular}{clcl}
$3$&$010  $                  &                $13$&$0102012101202  $ \\
$6$& $010212  $              &              $14$&$01020120212010  $ \\
$8$& $01020121  $             &            $15$&$010201210120212  $\\
$9$& $010212010  $             &           $16$&$0102012021201020$\\
$10$&$0102012101  $           &              $17$ &$01202120102120210 $\\
$11$&$01020120212  $         &&
 \end{tabular}
\caption{Small irreducibly square-free words. There are no examples for the lengths
4,5,7 and 12.} \label{tab:small}
\end{center}
\end{table}

For $n \ge 18$, we rely on  the morphism $\uno$ in order to have solutions for the
lengths $n \equiv p \pmod{17}$ for $p=0,1, \ldots, 16$.

\medskip

\noindent \emph{Claim~A.} Let $w$ be a nonempty suffix of $\uno(1)$ or $\uno(2)$ of length $|w| < 17$.
Then the word $w\Uuno$ is square-free (but not necessarily irreducibly square-free).

\smallskip

The word $w$ is a suffix of exactly one of the words $\uno(a)$, $a\in \Sigma$.
Suppose there is a square in $w\Uuno$ and assume that $w$ is of minimal length with this
property.
Then $w\Uuno$ has a prefix $uu$ for $u=w\uno(x)z$ for some words $x$ and  $z$
with $|z|<17$ (when $|x|$ is chosen to be maximal).
Hence $uu= w\uno(x)zw\uno(x)z$, and so $zw=\uno(a)$ for $a \in \Sigma$.
Therefore $z$ is nonempty.
By the alignment property, $uu$ must be followed in $w\Uuno$ by the word $w$. This delivers a square in
$\Uuno$, namely $\uno(x)zw\uno(x)zw=\uno(xaxa)$; a contradiction since $\Uuno$ is
square-free. This proves Claim~A.

\smallskip

Clearly, there are irreducibly square-free words of lengths $n \equiv 0 \pmod{17}$,
since we can take a prefix of $\Uuno$ of length $n/17$ and apply $\uno$ to it. Next we
extend $\Uuno$ to the left by considering words of the form $uw$, where $w$ is a
prefix of $\Uuno$.

\medskip

\noindent \emph{Claim~B.} The words  $121\Uuno$ and $0102\Uuno$ are square-free.

\smallskip

First, the words $121\uno(0)$ and $0102\uno(0)$ are not factors of $\Uuno$, since
$\uno$ has the alignment property and the given words are not suffixes of any $\uno(a)$, $ a\in
\Sigma$. Therefore, if $121\Uuno$ contains a square, then the square must be a prefix
$21v$ of  $21\Uuno$ for some $v$ (and  $1\Uuno$ is square-free by Claim~A).

Assume that $21v=21u21u$ where $v=u21u$ is a prefix of $\Uuno$. 
Now, $|v| > |\varphi(a)| = 17$, since $21$ is not followed by
the first letter of $u$ in any $\varphi(a)$, i.e.,
$u21=\uno(z)\uno(1)$ for some $z$, since only $\uno(1)$ ends in $21$. 
We have then that $\uno(1)=y21$ and $u=\uno(z)y$.
This means that the square $21v = 21\uno(z1z)y$ is necessarily
continued by the rest of $\uno(1)$, i.e., by $21$, giving a prefix $v21=u21u21$ of
$\Uuno$; a contradiction, since $\Uuno$ is square-free.

In the case of $0102\Uuno$, Claim~A guarantees that $102\Uuno$ is square-free. For the
full prefix $0102$, the claim follows since the prefix $01020120$ of $0102\Uuno$ does
not occur in $\Uuno$. This proves Claim~B.

\smallskip

The \emph{special words} $w_i$ of Table~\ref{tab:17} are chosen such that
\begin{itemize}
\item[(iii)]
$|w_i| = i$,
\item[(iv)]
$w_i\Uuno$ is square-free. By Claim~A, this follows for $i=1,2,4,5$ and $10$. By
Claim~B, the claim follows for the other cases.

\item[(v)]
$w_i\uno(0)$ is irreducibly square-free (by a simple computer check) .

\end{itemize}
The words $w_i$, themselves, are not (and, indeed, cannot be) all irreducibly
square-free, but they are square-free.

\begin{table}[hbt]
\begin{center}
\begin{tabular}{lllc|lllc}
 $w_1=1$&&                                                  $w_9=121020121$\\
 $w_2=02$&&                                                $w_{10}=2021012102$\\
 $w_3=121$&&                                                            $w_{11}=10121020121$ \\
 $w_4=2102$&&                                                            $w_{12}=101202120121$\\
 $w_5= 12102$&&                                                         $w_{13}=0210121020121$\\
  $w_6=020121$&&                                                      $w_{14}=01021201020121$\\
  $w_7=2120102$&&                                                     $w_{15}=010201202120121$ \\
  $w_8=01020121$&&                                                   $w_{16}= 0201021201020121$\\
\end{tabular}
\end{center}
\caption{The special words with $w_i \equiv i \pmod{17}$. The words $w_i$ with $i > 1$ that
end in $121$ or $0102$ as called for by Claim~B.}\label{tab:17}
\end{table}

Finally, let $n=17k+i$. By Table~\ref{tab:small}, we can assume that $n \ge 18$. We
then choose $w_i$ from Table~\ref{tab:17}, and pick a prefix $\uno(v)$ of $\Uuno$ of length
$17k$. This creates an irreducibly square-free word $w_i\uno(v)$ of length $n$.
\end{proof}

\section{Problems on Longer Words to Dispose}

The property of being irreducibly square-free can be generalized to longer factors than just letters.
Let $w \in \Sigma^*$ be a square-free word with a factorization $w=w_1vw_2$ such that
both $w_1$ and $w_2$ are nonempty. We say that the (occurrence of the) factor $v$
is \emph{disposable} if also $w_1w_2$ is square-free.
If a finite or infinite square-free word $w$ does not have disposable factors of length $k$
then $w$ is called \emph{$k$-irreducibly square-free}.

\begin{example}
We show that $\tau^{2n}(0)$ is \emph{not}
2-irreducibly square-free, for all $n \ge 2$.
Indeed, 
\begin{align*}
\tau^2(0) &= 012021\\
\tau^2(1) &= 0121\\
\tau^2(2) &= 02	
\end{align*}
Now, for $n \ge 2$, the word $\tau^{2n}(0)$ has the suffix $121$ since $\tau^2(1)$ has this suffix and $\tau^2(0)$ ends with the letter $1$.
But a $2$-irreducibly square-free word cannot be of the form $w121$, since by removing
the pair $12$, we obtain a (square-free) prefix of $w1$ of $w121$.

However, the limit $\Thue=\tau^\omega(0)$ is 2-irreducibly
square-free. 
To see this, we consider the 5-th powers of the morphism $\tau$:
\begin{align*}
\tau^5(0) &= 012021012102012021020121012021012102012101202102\\
\tau^5(1) &= 01202101210201202102012101202102\\
\tau^5(2) &= 0120210121020121
\end{align*}
where the lengths of the images are $48, 32$ and $16$, respectively.
These images have a common prefix $p=012021$ (and even longer ones).
A computer check shows that the words $\tau^5(a)p$
are 2-irreducibly square-free for $a=1,2$. Moreover, deleting
an internal occurrence of a pair $cd$ from $\tau^5(0)p$ results in a
square-free word only for $cd=20$ and $cd=02$ that lie inside $p$.
This proves that the infinite word $\Thue$ is 2-irreducibly square-free. 
\end{example}

These considerations raise many problems.

\begin{problem}\label{ProbInf}
Given $k \ge 1$, does there exist an infinite ternary 
word that is $k$-irreducibly square-free?
\end{problem}

\begin{theorem}
Every infinite ternary square-free word $w$ does have an infinite number of integers~$k$ for which $w$ is not $k$-irreducibly square-free. 
\end{theorem}

\begin{proof}
We need only to consider repetitions of the first letter of $w$, say $w=auaw_0$.
Deleting the factor $ua$ gives $aw_0$, a (square-free) suffix of~$w$.
\end{proof}

We have seen that
Problem~\ref{ProbInf} has a positive solution for $k=1$ and $k=2$. For small values of $k$ 
a solution may be found using square-free
morphisms. E.g., the morphism
\begin{align*}
\alpha_3(0) &=0121012\\ 
\alpha_3(1) &=01020120212\\ 
\alpha_3(2) &=0102101210212 
\end{align*}
generates a $3$-irreducibly square-free word $\alpha^\omega(0)$.
This follows from the fact that $\alpha_3(ab)$ is 3-irreducibly
square-free for all different letters $a$ and $b$.

\begin{problem}\label{ProbFin}
Does there exist, for every $k$,  a bound $N(k)$ such that there exist
$k$-irreducibly square-free words of all lengths $n \ge N(k)$?
\end{problem}

Finally, we state a problem of the opposite nature:

\begin{problem}\label{ProbFin2}
Does there exist an infinite square-free word $w$ such that 
 $w$ is $k$-irreducibly square-free for no $k\ge 1$?
\end{problem}

\medskip

\noindent\textbf{Acknowledgements.} I would like to thank the referees of this journal
for their comments that improved the presentation of this paper.

\bibliography{bib-small}

\end{document}